\documentclass[reqno]{amsart}
\usepackage{amssymb, amsmath}
\usepackage{mathrsfs}
\usepackage{tikz-cd} 
\usepackage{hyperref}
\usepackage[normalem]{ulem}

\renewcommand{\leq}{\leqslant}
\renewcommand{\geq}{\geqslant}
\newcommand{\seq}{\subseteq}

\newcommand\restr[2]{{% we make the whole thing an ordinary symbol
  \left.\kern-\nulldelimiterspace % automatically resize the bar with \right
  #1 % the function
  \vphantom{\big|} % pretend it's a little taller at normal size
  \right|_{#2} % this is the delimiter
  }}

\makeatletter
\newdimen\rh@wd
\newdimen\rh@hta
\newdimen\rh@htb
\newbox\rh@box
\def\rh@measure#1{\setbox\rh@box=\hbox{$#1$}\rh@wd=\wd\rh@box \rh@hta=\ht\rh@box}

\def\widecheck#1{\rh@measure{#1}%
  \setbox\rh@box=\hbox{$\widehat{\vrule height \rh@hta width\z@ \kern\rh@wd}$}%
  \rh@htb=\ht\rh@box \advance\rh@htb\rh@hta \advance\rh@htb\p@
  \ooalign{$\vrule height \ht\rh@box width\z@ #1$\cr
           \raise\rh@htb\hbox{\scalebox{1}[-1]{\box\rh@box}}\cr}}
\makeatother%%% End widecheck command
\newtheorem{theorem}{Theorem}[section]

\newtheorem{lemma}[theorem]{Lemma}
\newtheorem{claim}[theorem]{Claim}

\theoremstyle{definition}
\newtheorem{definition}[theorem]{Definition}

\newtheorem{remark}[theorem]{Remark}
\newtheorem{example}[theorem]{Example}
\newtheorem*{notation}{Notation}

\newcommand{\R}{\mathbb{R}}
\newcommand{\Z}{\mathbb{Z}}

\newcommand{\N}{\mathbb{N}}

\newcommand{\card}{\times} %% this is the (direct) product, traditionally denoted \boxplus and called "cardinal sum"

\DeclareMathOperator{\Spec}{{\rm Spec}}
\DeclareMathOperator{\MinS}{{\rm Min}}
\DeclareMathOperator{\MaxS}{{\rm Max}}
\newcommand{\m}{{\mathfrak{m}}} % a (maximal) ideal
\renewcommand{\a}{{\mathfrak{a}}} % a (minimal) prime  ideal
\renewcommand{\b}{{\mathfrak{b}}} % another (minimal) prime  ideal
\newcommand{\p}{{\mathfrak{p}}} % a  prime  ideal
\newcommand{\q}{{\mathfrak{q}}} % another  prime  ideal
\DeclareMathOperator{\VV}{\mathbb{V}} % vanishing locus, zero sets
\DeclareMathOperator{\Ss}{\mathbb{S}} % supports
\DeclareMathOperator{\Pol}{{\rm Pol}} %% Polars
\DeclareMathOperator{\pPol}{\Pol_{\mathrm{p}}} %% Princ Polars
\DeclareMathOperator{\C}{\rm C}

\DeclareMathOperator{\K}{\rm K} %characteristic functions
\DeclareMathOperator{\Ba}{\mathscr{B}} %Boolean algebra generated by zerosets

\DeclareMathOperator{\Cp}{\mathrm{Cp}} % closed and open sets
\DeclareMathOperator{\PH}{\mathscr{P}}

\newcommand{\up}[1]{{\uparrow}{#1}}

\newcommand{\orth}[2]{{{#1}{\perp}{#2}}}
\newcommand{\pperp}{{{\perp}{\perp}}}

\newcommand{\U}{{\mathbf U}} %Archimedean l-groups with strong unit
\title[Projectable hulls and minimal spectra]{From Freudenthal's Spectral Theorem   to projectable hulls of unital Archimedean lattice-groups, through compactifications of  minimal spectra}
\author[R. N. Ball]{Richard N. Ball}
\author[V. Marra]{Vincenzo Marra}
\author[D. McNeill]{Daniel McNeill}
\author[A. Pedrini]{Andrea Pedrini}
\address[R. N. Ball]{Department of Mathematics,
University of Denver,
Denver CO 80112,
U.S.A.}
\address[V. Marra, A. Pedrini]{Dipartimento di Matematica {\sl Federigo Enriques}, Universit\`a degli Studi di Milano, via Cesare Saldini 50, 20133 Milano, Italy.}
\address[D. McNeill]{Dipartimento di Scienze Teoriche e Applicate, Universit\`a degli Studi dell'Insubria, Via Mazzini 5, 21100 Varese, Italy.}
\email[R. N. Ball]{rball@du.edu}
\email[V. Marra, A. Pedrini]{\{vincenzo.marra, andrea.pedrini\}@unimi.it}
\email[D. McNeill]{danmcne@gmail.com}

\thanks{2010 {\it Mathematics Subject Classification.
}
Primary: 06F20. Secondary: 54D35.}
\keywords{Freudenthal's Spectral Theorem, Riesz space, Lattice-ordered group, Archimedean property, Yosida representation, Principal projection property, Projectable hull, Spectral space, Maximal ideal, Minimal ideal, Zero-dimensional space, Compactification.}
\begin{document}
\maketitle

\begin{abstract}We use a landmark result in the theory of Riesz spaces --- Freudenthal's 1936 Spectral Theorem --- to canonically represent any Archimedean  lattice-ordered  group $G$ with a strong unit  as a (non-separating) lattice-group of real valued continuous functions on  an appropriate $G$-indexed  zero-di\-men\-sional compactification $w_{G}Z_{G}$ of  its space  $Z_{G}$ of \emph{minimal} prime ideals. The two further ingredients needed to establish  this representation are the    Yosida representation of $G$ on its space $X_{G}$ of \emph{maximal} ideals, and the well-known continuous surjection of $Z_{G}$ onto $X_{G}$. We then establish our main result by showing that the inclusion-minimal extension of this representation of $G$ that separates the points of $Z_{G}$ --- namely, the sublattice subgroup of $\C{(Z_{G})}$ generated by the image of $G$ along with all  characteristic functions of clopen (closed and open) subsets of $Z_{G}$ which are determined by elements of $G$ --- is precisely the classical projectable hull of $G$. Our main result thus reveals a fundamental relationship between projectable hulls and minimal spectra, and  provides  the most direct and explicit construction of projectable hulls to date. Our  techniques do require the presence of a strong unit.
\end{abstract}

\section{Introduction}\label{s:intro}
In 1936, Freudenthal proved his well-known Spectral Theorem  \cite{Freudenthal1936} for Riesz spaces (real linear spaces with a compatible lattice order) with motivations coming from the theory of integration. (See \cite[40.2]{LuxZaan71} for a handbook treatment.)

 In its basic version, the theorem asserts that any element of a Riesz space $R$ with a strong unit $u$ and the principal projection property may be uniformly approximated,  in the norm that $u$ induces on $R$, by abstract characteristic functions --- ``components of  the unit  $u$''.  See Subsection \ref{ss:pol} for more details. Freudenthal's theorem led to a considerable amount of research on Riesz spaces and their generalisations, the lattice-ordered Abelian groups that concern us here, and which we  call \emph{$\ell$-groups} for short. (For background we refer to \cite{LuxZaan71, Darnel, Glass}.) One main line of research concentrated on extending one given structure $G$ to a minimal completion  that enjoys the principal projection property, where Freudenthal's theorem therefore applies. Such an  extension  is called the \emph{projectable hull} of $G$; please see Subsection \ref{ss:esspol} for details.

In 1973, Conrad \cite{Conrad73} proved the existence and  uniqueness of  projectable hulls of (a class of lattice-groups more general than) Archimedean $\ell$-groups, using his previous construction in \cite{Conrad71} of the \emph{essential closure}  of such an $\ell$-group --- the largest extension of the structure that is \emph{essential}, in the sense recalled in Subsection \ref{ss:esspol}. At about the same time, Chambless \cite{Chambless72} exhibited a different construction of the projectable hull based on direct limits; cf.\ also Bleier's construction in \cite{Bleier74}. Here we present a new construction of the projectable hull of an Archimedean $\ell$-group equipped with a strong order unit $u$ --- an element whose multiples eventually dominate any other element in the $\ell$-group --- that does not use direct limits, nor essential closures. Our construction exposes instead the intimate connection between projectable hulls and   zero-dimensional compactifications of  spectral spaces of minimal prime ideals. Closing the circle of ideas beginning with Freudenthal, to establish this connection we will need to apply his Spectral Theorem at a key step of the construction. We now recall some standard notions, and introduce notations that will remain in force throughout the paper.

Throughout,  all lattice-ordered groups are Abelian, and referred to simply as $\ell$-groups for short. We write $\U$ for the category whose typical object is a pair $(G,u)$, where $G$ is an  $\ell$-group that is \emph{Archimedean} --- whenever $0\leq ng\leq h$ for $h,g \in G$ and all integers $n\geq 1$, then $g=0$ --- 
equipped with a distinguished (\emph{strong order}) \emph{unit} $u\in G$ --- an element $u\geq 0$   such that for all $g\in G$ there is an integer $n\geq 1$ such that $nu\geq g$. As morphisms, we take the   lattice-group homomorphisms (\emph{$\ell$-homomorphisms}) that are \emph{unital}, i.e.\ preserve the distinguished units.  It will transpire that our  techniques do require the existence of a strong unit, as opposed to the existence of a weak unit. Recall that a \emph{weak \textup{(}order\textup{)} unit} of $G$ is an element $w\in G^{+}$ such that for each $g \in G$,  $w\wedge |g|=0$ implies  $g=0$. Here,  $|g|:=(g\vee 0) + (-g\vee 0)$ is the \emph{absolute value} of $g$.

By an \emph{ideal} in an $\ell$-group we mean, as usual, a sublattice subgroup $I$ of $G$ that is \emph{order-convex}: whenever $a,c\in I$, $b\in G$, and $a\leq b\leq c$, then $b\in G$. Ideals are exactly the kernels of (unital) $\ell$-homomorphisms, i.e.\ the morphisms in the unrestricted category of abelian $\ell$-groups, and the usual homomorphism theorems hold. An ideal $\p$ of $G$ is \emph{prime} if, and only if,  it is proper ($\p\neq G$) and the quotient $\ell$-group $G/\p$ is totally ordered. A prime ideal is \emph{maximal}  if it is   inclusion-maximal --- equivalently, if $G/\p$ is non-trivial and \emph{simple}, i.e.\ it has no non-trivial proper ideals. Ideals that are inclusion-maximal are automatically prime. A prime ideal is  \emph{minimal} if it is inclusion-minimal. For any unital $\ell$-group  $(G,u)$, we denote by $\MaxS{G}$ the collection of its maximal (prime) ideals, and by $\MinS{G}$ the collection  of its minimal prime ideals. 
We  topologize both $\MaxS{G}$ and $\MinS{G}$  using the \emph{spectral}, or \emph{Zariski} topology. The closed sets for this topology are given by  subsets of the form
\begin{align*}
\VV_{M}{(A)}:=\{\mathfrak{m}\in \MaxS{G} \,\mid\, \mathfrak{m}\supseteq A\} 
\end{align*}
and
\begin{align*}
\VV_{m}{(A)}:=\{\mathfrak{p}\in \MinS{G} \,\mid\, \mathfrak{p}\supseteq A\},
\end{align*}
as $A$ ranges over arbitrary subsets of $G$. The resulting topological spaces  are called the \emph{maximal} and \emph{minimal prime spectrum} of $G$, respectively. The topology 
on $\MaxS{G}$ is also called the \emph{hull-kernel} topology, because  it agrees with the classical  hull-kernel topology for rings of continuous functions \cite{GillmanJerison76}, \textit{mutatis mutandis}. Accordingly, we  call $\VV_M{(A)}$ (or $\VV_m{(A)}$)  the \emph{zero set} of $A$ (on the appropriate space), and its complement --- denoted by $\Ss_{M}{(A)}$ (or $\Ss_{m}{(A)}$) --- the \emph{support} of $A$. It is known \cite[10.2.1]{BKW} that the collections 
\begin{equation*}
 \{\VV_{M}(\{g\})\}_{g\in G}\qquad\text{and}\qquad\{\VV_{m}(\{g\})\}_{g\in G}
\end{equation*}
form a closed base for $\MaxS{G}$ and $\MinS{G}$, respectively. Throughout we write $\VV_{M}{(g)}$ in place of $\VV_{M}{(\{g\})}$, and similarly for $\VV_{m}$, $\Ss_{M}$, and $\Ss_{m}$. 

The space $\MaxS{G}$ is a Hausdorff space that is compact precisely because of the assumption that $G$ has a (strong) unit $u$; see \cite[10.2.5]{BKW}. The space $\MinS{G}$ is a Hausdorff zero-dimensional space that need not be compact  \cite[10.2.1]{BKW}. Whether it is or not has nothing to do with the existence of a strong unit, but rather with  complementation properties of the lattice $G^{+}:=\{g\in G \mid g\geq 0\}$, the \emph{positive cone} of $G$; see Section \ref{s:repG}.  

\begin{notation}For the rest of this paper, we let $(G,u)$ denote a $\U$-object, and set
\begin{align*}
 X_{G}&:=\MaxS{G},\\
 Z_{G}&:=\MinS{G}.
\end{align*} 
\end{notation}
If $X$ is any topological space, always at least Tychonoff, we write $\C{(X)}$ for the $\ell$-group of continuous functions $X\to\R$ under pointwise operations. If $X$ is compact, the function $1_{X}$ constantly equal to $1$ over $X$ is a strong unit of $\C{(X)}$ by the Extreme Value Theorem. We  always tacitly consider $\C{(X)}$ endowed with the distinguished unit $1_{X}$, and hence as a $\U$-object when $X$ is compact Hausdorff. The classical \emph{Yosida representation} \cite{Yosida41} of $(G,u)$ yields a canonical unital lattice-group embedding  $\,\widehat{\cdot}\colon G \hookrightarrow \C{(X_G)}$; details are recalled in Subsection \ref{ss:Yosida}.

It is well known that $Z_{G}$ is canonically thrown onto $X_G$, as follows.  Given $\a \in Z_{G}$, a standard argument \cite[27.4]{LuxZaan71} shows that, by virtue of the presence of the (strong) unit $u$,   there exists at least one $\m_{\a} \in X_G$ such that $\a \seq \m_{\a}$. Since the prime ideals of $G$ form a \emph{root system} under set-theoretic inclusion \cite[2.4.3]{BKW} --- that is, the set of prime ideals containing any given prime ideal is linearly ordered --- such an $\m_{\a}$ must be unique; in other words, the set $\up{\a}\cap X_G$ is a singleton, where $\up{\a}:=\{\b\seq G \mid \a \seq \b, \b \text{ a prime ideal}\}$. Hence there is a function
\begin{align}\label{eq:lambda}
\lambda \colon Z_{G}\twoheadrightarrow X_{G}
\end{align}
defined by
\begin{align}\label{eq:action}
\a \in Z_{G}\overset{\lambda}{\longmapsto} \m_{\a} \in X_G.
\end{align}
By \cite[10.2.5]{BKW}, the map $\lambda$ is continuous, and is a surjection by the standard fact that each prime ideal contains a minimal prime ideal \cite[2.4.5]{BKW}. 

\begin{remark}The strong unit is crucial here. There exist non-trivial Archimedean  (and even Dedekind-complete) $\ell$-groups with a weak unit and no maximal ideal at all; see \cite[27.8]{LuxZaan71}. 
Thus, while the  Yosida representation has an important extension to $\ell$-groups with a weak unit \cite{Yosida42, HagerRobertson77}, the existence of the map (\ref{eq:lambda}), a key ingredient to our construction, does require a strong unit. A generalisation of our results to the case of  weak  units would surely be of interest, but it would require substantial modifications. \end{remark}

Composition of the map $\lambda$ with the Yosida representation of $G$ embeds $G$ as a unital sublattice-subgroup (\emph{$\ell$-subgroup}) into 
$\C{(Z_{G})}$: one sends $g\in G$ to $\widehat{g}\circ \lambda \colon Z_{G}\to\R$. The assignment is injective because $\lambda$ is surjective.
In  Section \ref{s:repG}  this observation is considerably strengthened. It turns out that $G$ determines a specific zero-dimensional compactification of its minimal spectrum which we denote $w_{G}Z_{G}$.  The clopen subsets of $w_{G}{Z_{G}}$, we will see, are (finitely)
 generated from the clopen subsets of $Z_{G}$ of the form $\VV_m(g)$. (Please see   Section \ref{s:repG} for details on this compactification.) We prove
in Theorem \ref{t:min-rep} that $G$ embeds as a unital $\ell$-subgroup of $\C{(w_{G}{Z_{G}})}$.    We will see that this stronger embedding of $G$ is in fact granted by Freudenthal's Spectral Theorem. Now, by the  Yosida theory (see again Subsection \ref{ss:Yosida}), the image of $G$ in $\C{(w_{G}{Z_{G}})}$  does not separate 
the points of the base space, unless $X_G$ and $w_{G}{Z_{G}}$ are homeomorphic. We can however consider a minimal extension of the image of $G$ inside $\C{(w_{G}{Z_{G}})}$ which separates the points. Indeed, since $w_{G}{Z_{G}}$ is zero-dimensional, there is a canonical such extension: we must adjoin to 
the image of $G$ all characteristic functions of clopen  subsets of $w_{G}{Z_{G}}$. We thereby obtain a unital embedding 
\begin{align}\label{eq:phintro}
\pi_{G} \colon G \hookrightarrow \PH{(G)},
\end{align}
where $\PH{(G)}$ denotes the unital $\ell$-subgroup of $\C{(w_{G}Z_{G})}$ generated by the  representation  of $G$ into  $\C{(w_{G}Z_{G})}$, together with all \emph{characteristic functions} $w_{G}Z_{G}\to \R$ --- the continuous maps with range contained in $\{0,1\}$. We now have the homeomorphism $\MaxS{\PH{(G)}}\cong w_G{Z_{G}}$. In Theorem \ref{t:decomposition} we show that the elements of $\PH{(G)}$ may be characterised amongst elements of $\C{(w_{G}{Z_{G}})}$ as those functions with the property that, for an appropriate finite partition of $w_{G}Z_{G}$ into clopens, they agree with the image of some element of $G$ locally at each clopen. Building on this 
we finally show in Theorem \ref{t:ph} that \eqref{eq:phintro} is the projectable hull of $(G,u)$, thus obtaining our main result. Summarising, we prove the existence of the projectable hull  of any  $\U$-object $(G,u)$ by exhibiting it as a natural substructure of $\C{(w_{G}Z_{G})}$, namely, $\PH{(G)}$.

\smallskip
Several intermediate  results in this paper admit a fuller development  of considerable potential  interest. We focus here on the proof of our main Theorem \ref{t:ph}, and postpone further results to future work.

\section{Preliminaries}\label{s:back}
\subsection{Polars and projection properties}\label{ss:pol}For the standard notions that we recall in this subsection, see \cite[Ch.\ 3, 6, and 11]{BKW}, together with \cite[Ch.\ 4, \S 24]{LuxZaan71}.
Given any $\ell$-group $A$, the elements $x,y \in A$ are \emph{orthogonal}, written $\orth{x}{y}$, if 
$|x|\wedge |y|=0$. For $T \seq A$, we set 
\begin{align*}
T^\perp:=\{x \in A\mid \orth{x}{y} \text{ for all } y \in T\};
\end{align*}
we  write $T^\pperp$ instead of $(T^\perp)^\perp$, and $x^\perp$ instead of $\{x\}^\perp$ for $x \in A$. 
A subset $S\seq A$ is a \emph{polar} if it satisfies $S=S^\pperp$, or equivalently, if there exists $T\subseteq A$ such that $S=T^\perp$. When necessary, we will denote polars computed in $A$ by $T^{\perp_{A}}$. We write
$\Pol{A}$ to denote  the set of polars of $A$. Under the inclusion order,  $\Pol{A}$ is a complete distributive lattice with  $A=0^\perp$ as maximum,  $\{0\}=A^\perp=0^\pperp$ as minimum, meets given by intersections, and joins given by $\bigvee S_i:=(\bigcup S_i)^\pperp$. It can be shown that $\Pol{A}$ is a complete Boolean algebra, with complementation given by the map $S\in\Pol{A}\mapsto S^\perp\in\Pol{A}$. In particular, for any subset $T\seq A$ we have $T^{\pperp\perp}=T^\perp$.

If $x\in A$, the set $x^\pperp$ is called the \emph{principal polar} generated by $x$. Then $x^\pperp\in\Pol{A}$, and $x^\pperp=\bigcap_{x \in S\in\Pol{A}} S$, that is, $x^\pperp$ is the inclusion-smallest polar containing $\{x\}$. We write $\pPol{A}$ to denote the set of principal polars of $A$; it is a sublattice of $\Pol{A}$, because of the identities
\begin{align}
(x\wedge y)^\pperp&=x^\pperp\cap y^\pperp\label{e:ppolmeet}\\
(x\vee y)^\pperp&=x^\pperp\vee y^\pperp,\label{e:ppoljoin}
\end{align}
which hold for each $x,y\in A^+$. Further, the minimum $0^\pperp$ of $\Pol{A}$ lies in $\pPol{A}$. However, the maximum $A=0^\perp$ of $\Pol{A}$ need not be a principal polar: in fact, this happens precisely when $A$ has a weak unit $w$, and in that case $A=w^\pperp$.  Even when $A$ has a weak unit, $\pPol{A}$ may fail to be a Boolean subalgebra of $\Pol{A}$, because the complement of a principal polar need not be principal.

An ideal $I\subseteq A$ is  \emph{closed}, or is a \emph{band}, if for each $S\subseteq I$ such that $\bigvee S$ exists in $A$, we have $\bigvee S \in I$. It can be shown that each polar is a band; for the converse, we have the important
\begin{lemma}\label{l:pol=band}An $\ell$-group $A$ is such that its polars coincide with its bands if, and only if, $A$ is Archimedean.
\end{lemma}
\begin{proof}\cite[11.1.10]{BKW}.
\end{proof}
A band $I \seq A$ is a \emph{projection band} if there is a product splitting $A \cong I \card I^\perp$. 
\begin{definition}[{Cf.\ \cite[24.8]{LuxZaan71}}]\label{d:projprop}An $\ell$-group $A$ is said to have the \emph{principal projection property}, or to be \emph{projectable}, if each principal band of $A$ is a projection band. Further, $A$ is said to have the \emph{projection property}, or to be \emph{strongly projectable}, if each band of $G$ is a projection band.
\end{definition}
We recall here a standard fact:
\begin{lemma}\label{l:pppimpliesarch}An $\ell$-group with the principal projection property must be Archimedean.  
\end{lemma}
\begin{proof}The (easy) proof for vector lattices given in \cite[24.9]{LuxZaan71} works for $\ell$-groups without changes.
\end{proof}
\begin{remark}\label{r:terminology}
Projection properties are a  classical topic in the theory of vector lattices, see \cite[Ch.\ 4]{LuxZaan71}. In the literature on $\ell$-groups, it is standard to call $A$ \emph{projectable} when each of its principal  polars is a cardinal summand (i.e.\ a factor of a product splitting) of $A$, and \emph{strongly projectable} when the same holds for all polars. Thus, we see from Lemmas \ref{l:pppimpliesarch} and \ref{l:pol=band} that an $\ell$-group $A$ has the  principal projection property if, and only if, it is  projectable in the present sense; and that it has the  projection property if, and only if, it is strongly projectable in the present sense. Cf.\ also \cite[7.5]{BKW}. This explains the alternative terminologies in Definition \ref{d:projprop}.  In the rest of this paper we shall use the terminology \emph{projectable}.
\end{remark}
 A \emph{component of  the unit $u$} is an element $\chi \in G$ such that $\chi\vee (u-\chi)=u$ and $\chi\wedge (u-\chi)=0$. It is well known that this  entails  the existence of a product splitting $G\cong \chi^{\perp}\times\chi^{\perp\perp}$. Conversely, if $G\cong A\times B$ in $\U$, then there is a unique  $\chi \in G$ --- namely, the image in $G$ of the unit of $B$ under the unital isomorphism $G\cong A\times B$ ---  that is a component of the unit $u$ such that $A\cong \chi^{\perp}$ and $B\cong \chi^{\perp\perp}$. We use these elementary facts  without further justification throughout.

Finally, we recall the version of Freudenthal's Spectral Theorem that we will use.
\begin{theorem}\label{t:freudenthal} Let $R$ be a Riesz space that is projectable and has a unit $u$. For $v \in R$, set $||v||_{u}:=\inf{\{\lambda\in\R\mid \lambda \geq 0 \text{ and } \lambda u \geq |v|\}}$. Then $||v||_{u}$ is a norm on $R$. For each $v\in R$ there is a sequence $\{c_{i}\}_{i\geq 1}\seq R$ of linear combinations of components of $u$ that converges  to $v$ uniformly in the norm $||\cdot||_{u}$.
\end{theorem}
\begin{proof}See 
 \cite[40.2]{LuxZaan71}.
 \end{proof}

\subsection{Essential extensions and the projectable hull}\label{ss:esspol}
A monomorphism $\iota \colon (G,u)$ $\hookrightarrow$ $(H,v)$  in $\U$ will be referred to as an \emph{extension} (of $G$ by $H$). The extension is \emph{essential} if whenever a $\U$-morphism $f\colon (H,v) \to (A,a)$ is such that the composition $f\circ\iota$ is monic, then $f$ is monic.
Amongst several well-known characterisations of essential extensions we shall use the following.
\begin{lemma}\label{l:ess}Let $\iota \colon (G,u)\hookrightarrow (H,v)$ be a monomorphism in $\U$. The following are equivalent. 
\begin{enumerate}
\item\label{i:ess1}  The extension $\iota$ is essential.
\item\label{i:ess2}  The map $\nu_{H}\colon P \in \Pol{H}\longmapsto P \cap \iota(G) \in \Pol{\iota(G)}$ is an isomorphism from the Boolean algebra of polars of $H$ onto that of $\iota(G)$. The inverse isomorphism is the map $\nu_{H}^{-1}\colon Q \in \Pol{\iota(G)} \longmapsto Q^{\pperp_{H}} \in \Pol{H}$.
\item\label{i:ess3} For each $y \in H$ with $y >0$ there is $x \in G$ with $0 <\iota(x) < ny$ for some integer $n >0$.
\end{enumerate}
\end{lemma}
\begin{proof}See  \cite[Prop.\ 3.1 and Thm.\ 3.7]{Conrad71} and \cite[\S 2]{Conrad73}.
 \end{proof}
% %
\begin{definition}\label{d:ph}An essential extension $\epsilon\colon (G,u) \hookrightarrow (K,w)$  in $\U$ is said to be a \emph{projectable hull} if $K$ is projectable, and whenever $\iota\colon (G,u) \hookrightarrow (H,v)$ is another essential extension with $H$ projectable, there exists an injective $\ell$-homomorphism $\varphi\colon (K,w)\rightarrow (H,v)$ in $\U$ that makes the following diagram commute.
\[
 \begin{tikzcd}
(G,u)   \arrow[hookrightarrow]{r}{\epsilon}\arrow[hookrightarrow]{dr}{\iota}&(K,w)\arrow[rightarrow, dashed]{d}{\varphi} \\
&(H,v)
 \end{tikzcd}
\]
\end{definition}
It turns out that the $\ell$-homomorphism $\varphi$ in the preceding definition  is automatically an essential extension. Also note that a projectable hull is unique up to an isomorphism in $\U$. 
\begin{remark}\label{rem:generalisations}Through the general treatment in \cite{Ball82}, hulls related to projectability properties can and have been fruitfully investigated at the level of all  lattice-ordered (not necessarily Abelian) groups, with no assumption on the existence of units. In particular, any lattice-ordered group turns out to have a strongly projectable hull in this generalised sense, \cite[Thm.\ 2.25]{Ball82}, which agrees with the usual one in the representable case.
\end{remark}
\subsection{The Yosida representation: the case of strong units}\label{ss:Yosida}
For $X$ a topological space, recall that a subset  $S\subseteq \C{(X)}$ is said to  \emph{separate  the points of $X$} if for any $x\neq y \in X$ there is $f \in S$ with $f(x)\neq f(y)$. The next result summarises the classical Yosida representation; everything is rooted and essentially proved in \cite{Yosida41}.
\begin{theorem}[The Yosida Representation]\label{t:yosida} Recall that $(G, u)$ is a $\U$-object with maximal spectral space $X_{G}$.
\begin{enumerate}
\item[(a)]For each $\mathfrak{m}\in X_G$, there exists a \emph{unique} monomorphism 
\begin{align*}%\label{e:hoelder}
\iota_{\mathfrak{m}}\colon (G/\mathfrak{m},u/\mathfrak{m})\hookrightarrow (\R,1)
\end{align*}
in $\U$. Upon setting
\begin{align*}%\label{e:hoelderfunct}
\widehat{g}(\mathfrak{m}):=\iota_{\mathfrak{m}}(g/\mathfrak{m}) \in \R\,,
\end{align*}
each $g \in G$ induces a  function 
\begin{align*}
\widehat{g}\colon X_G\to \R
\end{align*} 
that  is continuous with respect to the spectral topology on the domain and the Euclidean topology on the co-domain. 
\item[(b)]The map
\begin{align*}
\widehat{\cdot} \,\,\colon (G,u) \longrightarrow (\C{(X_G)},1_{X_G})
\end{align*}
given by \textup{(a)} is a monomorphism in $\U$ whose image $\widehat{G}\subseteq \C{(X_G)}$ separates the points of $X_G$.
\item[(c)] $X_G$ is unique up to a unique homeomorphism with respect to its properties. More explicitly, if $Y$ is any compact Hausdorff space, and $e\colon (G,u)\hookrightarrow (\C{(Y)},1_{Y})$ is any monomorphism in $\U$ whose image $e(G)\subseteq \C{(Y)}$ separates the points of $Y$, then there exists a unique homeomorphism  $f \colon Y \to X_G$ such that $(e(g))(y) = \widehat{g} (f(y))$ for all $g\in G$ and $y \in Y$.
\end{enumerate} 
\end{theorem}
\begin{remark}\label{r:W}For the more general Yosida representation in the category of $\ell$-groups equipped with a \emph{weak} unit, see the standard reference  \cite{HagerRobertson77}.
\end{remark}
\begin{remark}Let us explicitly observe that components of the unit $1_{X_{G}}$ in $\C{(X_G)}$ are precisely the characteristic functions $X_{G}\to\R$, i.e.\ the continuous functions with range contained in $\{0,1\}$.
\end{remark}

\section{Representing an $\ell$-group on its minimal spectrum}\label{s:repG}
 Recall that $Z_{G}$ denotes the minimal spectral space of the $\U$-object $(G,u)$. We  show in this section that $G$ may be represented as an $\ell$-subgroup of $\C{(w_{G}{Z_{G}})}$ for the zero-dimensional Wallman compactification $w_{G}Z_{G}$ of $Z_{G}$, which we  describe  below.  
 Before dealing with the general case, let us pause to recall that compactness of $Z_{G}$ is  equivalent to a complementation property of $G$.
\begin{definition}[{\cite{ConradMartinez90bis}}]\label{d:complemented}An $\ell$-group $A$ is \emph{complemented} if for each $x \in A$ there exists $y \in A$ such that $|x|\wedge |y|=0$ and $|x|\vee |y|\neq 0$ is a weak unit of $A$.
\end{definition}
\noindent Throughout we write ${\cdot}{\setminus}{\cdot}$ for set-theoretic difference. We recall $\Ss_{m} = \MinS{A} \setminus \VV_{m}(g)$.
\noindent Conrad and Martinez \cite{ConradMartinez90} proved the equivalence of (i) and (ii) below.
\begin{lemma}\label{l:conmar}For an $\ell$-group $A$, the following are equivalent.
\begin{enumerate}
\item[(i)]\label{i:comp} $A$ is complemented.
\item[(ii)]  $\MinS{A}$ is compact.

\item[(iii)] There exists a weak unit in $A$, the lattice $\pPol{A}$ is bounded, and the inclusion map $\pPol{A}\hookrightarrow \Pol{A}$ is a homomorphism of Boolean algebras.
\item[(iv)] The distributive lattice $(\{\VV_{m}(g)\}_{g\in A},\cap,\cup)$ is a Boolean algebra, i.e.\ for all $g\in A$ there is $f\in A$ such that $\Ss_{m}(g)=\VV_{m}(f)$.
\end{enumerate}
\end{lemma}
\begin{proof}(i $\Leftrightarrow$ ii) is \cite[2.2]{ConradMartinez90}. 
(i $\Leftrightarrow$ iii) is also stated in passing in \cite{ConradMartinez90}; its proof is an elementary application of  (\ref{e:ppolmeet}--\ref{e:ppoljoin}).\\
To prove (iii $\Leftrightarrow$ iv) we recall that $\bigcap_{\p\in\MinS{G}}\p=\{0\}$ and that a prime $\p\in\Spec{A}$ is minimal if and only if for all $g\in\p$ there is an element $h\not\in\p$ with $h\perp g$ (see \cite[3.4.13]{BKW}). We first show $\VV_{m}(g) = \VV_{m}(g^{\pperp})$ for every $g\in A$. Since $g\in g^{\pperp}$, $\VV_{m}(g^{\pperp})\seq\VV_{m}(g)$. Let $\p\in\VV_{m}(g)$. Since $g\in\p$ and $\p\in\MinS{A}$, there exists $h\not\in\p$ such that $g\perp h$, whence $h\in g^{\perp}$. Let $f\in g^{\pperp}$. Then $f\in\p$, because $|f|\land|h|=0$ and $\p$ is prime. This ensures $\VV_{m}(g^{\pperp})\seq\VV_{m}(g)$.\\
(iv $\Rightarrow$ iii) Let $g\in A$ and $f\in A$ such that $\Ss_{m}(g) = \VV_{m}(f)$. Then each $\p\in\MinS{A}$ is contained either in $\VV_{m}(g)$ or in $\VV_{m}(f)$. We show that $g^{\pperp} = f^{\perp}$.\\
  Let $a\in f^{\perp}$. By primality, $a\in\p$ for each $\p\in\VV_{m}(g)=\Ss_{m}(f)$, and hence $\VV_{m}(g)\seq\VV_{m}(a)$. For each $b\in g^{\perp}$, $b\in\p$ whenever $g\not\in\p$, whence $b\in\p$ for each $\p\in\VV_{m}(f)$. Therefore, $|a|\land |b|\in\p$ for each $\p\in\MinS{A}$, and hence $|a|\land |b|=0$. This ensures $a\in g^{\pperp}$, whence $g^{\pperp}\supseteq f^{\perp}$.\\
  For the other inclusion, let $a\in g^{\pperp}$. By way of contradiction, suppose $|a|\land|f|=c>0$. Then there exists $\p_{c}\in\MinS{A}$ such that $c\not\in\p_{c}$. As a consequence, $a\not\in\p_{c}$, $f\not\in\p_{c}$, and $g\in\p_{c}$. By minimality of $\p$, there exists $b\not\in\p_{c}$ such that $b\perp g$. Since $|a|\land|b|=0\in\p_{c}$, then $a\in\p_{c}$, a contradiction. Hence we proved $g^{\pperp}\seq f^{\perp}$.\\
 As a consequence, $g^{\pperp\perp} = f^{\pperp}$, whence the principal polar $f^{\pperp}$ is the complement of $g^{\pperp}$ in $\pPol{A}$. Moreover, by taking $g=0$, there exists $w\in A$ such that $A=0^{\pperp\perp} = w^{\pperp}$. Then $\pPol{A}$ is bounded and a Boolean algebra, and, as a direct consequence of the definition, $w$ is a weak unit for $A$.\\
 (iii $\Rightarrow$ iv) Let $g\in A$ and $f\in A$ such that $f^{\pperp}$ is the complement of $g^{\pperp}$ in $\pPol{A}$. Then $g^{\pperp\perp} = f^{\pperp}$, and $g^{\pperp} = f^{\perp}$. We show $\VV_{m}(f)=\Ss_{m}(f^{\perp})$, whence $\VV_{m}(f)=\Ss_{m}(g^{\pperp})=\Ss_{m}(g)$. Let $\p\in\VV_{m}(f)$, then $f\in\p$. Since $\p$ is minimal, there is $h\in A$ such that $h\perp f$ and $h\not\in\p$. Hence $h\in f^{\perp}$ and $\p\not\in\VV_{m}(f^{\perp})$. This ensures $\VV_{m}(f)\seq \Ss_{m}(f^{\perp})$. For the other inclusion, let $\p\in \Ss_{m}(f^{\perp})$ and $h\not\in\p$ such that $h\perp f$. Since $|h|\land|f|=0$ the primality of $\p$ ensures $f\in\p$, whence $\p\in\VV_{m}(f)$. Therefore $\Ss_{m}(f^{\perp})\seq\VV_{m}(f)$ and the proof is complete.
\end{proof}
\begin{remark}\label{r:othermincomp}Versions of Lemma \ref{l:conmar} for commutative rings, distributive lattices, and vector lattices were proved in \cite[3.4]{HenJeri65}, \cite[Prop.\ 3.2]{Speed69}, and \cite[37.4]{LuxZaan71}, respectively.
\end{remark}
We now turn to the $w_{G}$-compactification. Recall \cite[4.4(a)]{PorterWoods88} 
that a \emph{Wallman base} of a Hausdorff space $X$ is a base $\mathscr{L}$ of closed sets for $X$ that is stable under finite intersections and unions (and thus contains, in particular, $\varnothing$ and $X$),  is such that if $A\in \mathscr{L}$ and $x\in X\setminus A$ then there is $B\in \mathscr{L}$ with $x\in B$ and $A\cap B=\varnothing$, and is such that for $A,B\in \mathscr{L}$ satisfying  $A\seq X\setminus B$  there exist $C,D\in \mathscr{L}$ with $A\seq X\setminus C\seq D \seq X\setminus B$. Given such a base $\mathscr{L}$, let $w_{\mathscr{L}}X$ denote the collection of inclusion-maximal lattice filters of $\mathscr{L}$. The collection of sets $\{\mathscr{F}\in w_{\mathscr{L}}X\mid A \in \mathscr{F}\}$, as $A$ ranges in $\mathscr{L}$, is a base for the closed sets of a topology on $w_{\mathscr{L}}X$. With this topology, $w_{\mathscr{L}}X$ is compact \cite[4.4(d)]{PorterWoods88}. Given $x \in X$, set $\mathscr{U}_{x}:=\{A\in \mathscr{L}\mid x \in A\}$.  Then $\mathscr{U}_{x}\in w_{\mathscr{L}}X$, and the map
\begin{align}\label{eq:wallman}
X &\longrightarrow w_{\mathscr{L}}X\\
x \in X &\longmapsto \mathscr{U}_{x}\in w_{\mathscr{L}}X\nonumber
\end{align}
is a dense embedding, called the \emph{Wallman compactification} of $X$ induced by $\mathscr{L}$.

For any zero-dimensional Hausdorff space $X$, we note that any collection of clopen sets $\mathscr{B}$ of $X$ which also forms a closed base for the space, is, almost trivially, a Wallman base for the space, and that the generated Wallman compactification $w_{\mathscr{B}}X$ is a zero-dimensional Hausdorff space \cite[4.7(b)]{PorterWoods88}.

Furthermore, we note that in the case where $X$ is zero-dimensional and $\mathscr{B}$ is a Wallman base for $X$ consisting of clopen sets, the (bounded) functions from the uniformly closed subring of $\C{(X)}$ generated by the characteristic functions for the Boolean algebra of clopen sets generated by $\mathscr{B}$ are exactly the functions which extend to $w_{\mathscr{B}}X$ \cite[4I \& 4J]{PorterWoods88}.

%% For any space $X$ we write $\Cp{X}$ to denote the Boolean algebra of clopen sets of $X$.  Then, since $Z_{G}$ is zero-dimensional, $\Cp{Z_{G}}$ is a Wallman base of $Z_{G}$ by \cite[4.7(b)]{PorterWoods88}. The associated compactification will be denoted
%% %
%% \begin{align*}
%%  w_{G}Z_{G},
%% \end{align*}
%% %
%% where $w_{G}Z_{G}:=w_{\Cp{Z_{G}}}Z_{G}$ is the Stone space of maximal ideals of the Boolean algebra $\Cp{Z_{G}}$, and hence is a compact Hausdorff zero-dimensional  space.
%%  Observe that, by construction, we have the isomorphism of Boolean algebras
%% %
%% \begin{align}\label{eq:beta0Cp}
%% \Cp{Z_{G}}\cong \Cp{w_{G}Z_{G}}.
%% \end{align}
%% %
%% Relatedly, the $w_{G}$-compactification  is canonical in that it is the largest zero-di\-men\-sional compactification of $Z_{G}$ \cite[4.7(c)]{PorterWoods88}, and may also be characterised as follows. For any space $X$, by a \emph{characteristic function} on $X$ we mean a continuous map $X\to \R$ whose range is contained in $\{0,1\}$. 

%% We write $\K_G{(X)}$ for the collection of characteristic functions on $X$. If $X$ is zero-dimensional, then $w_{G}(X)$ is the essentially unique zero-dimensional compactification  $Y$ of $X$ such that each  characteristic function on $X$ admits a continuous extension to a characteristic function on $Y$; see \cite[4.7(f)]{PorterWoods88}. 

We next identify $G$ with its Yosida representation $\widehat{G}\subseteq \C{(X_{G})}$, as given by Theorem \ref{t:yosida}. Recall the map $\lambda\colon Z_{G}\twoheadrightarrow X_{G}$ as in (\ref{eq:lambda}--\ref{eq:action}).  If $\widehat{g}\in \widehat{G}$, the assignment
\begin{align}\label{eq:lambdapb}
\widehat{g}\in \widehat{G}\overset{\mu}{\longmapsto} \widehat{g}\circ\lambda\in \C{(Z_{G})}
\end{align}
yields a unital homomorphism of $\ell$-groups $\mu \colon \widehat{G}\to \C{(Z_{G})}$, and a straightforward  computation confirms that $\mu$ is injective because $\lambda$ is surjective. We therefore obtain a  representation of $G$ as 
\begin{align}\label{eq:mu-rep}
\mu(\widehat{G})\subseteq \C{(Z_{G})}.
\end{align}
\begin{remark}\label{r:zerosets}
We notice that the Yosida representation $\widehat{g}$ of $g$ is such that, for every $\m\in X_{G}$, $\widehat{g}(\m)=0$ if, and only if, $g\in\m$. This ensures the inclusion of zero sets
\begin{equation*}
\VV_{m}(g)\seq\left\{\p\in Z_{G}\mid \mu(\widehat{g})(\p)=\widehat{g}(\lambda(\p))=0\right\}.
\end{equation*}
Hence
\begin{equation*}
\lambda^{-1}(\Ss_{M}{(g)})=\Ss{(\mu(\widehat{g}))}\seq\Ss_{m}{(g)},
\end{equation*}
where $\Ss{(\mu(\widehat{g}))}:=Z_{G}\setminus \mu(\widehat{g})^{-1}(0)$ is the support of $\mu(\widehat{g})$ at $Z_{G}$, and $\lambda^{-1}(\Ss_{M}{(g)})$ is the preimage of $\Ss_{M}{(g)}$ under $\lambda$.
\end{remark}
We consider now the Boolean algebra $\Ba{(Z_{G})}$ of subsets of $Z_{G}$ generated by the collection $\{\VV_m{(g)}\}_{g\in G}$. By \cite[10.2.1]{BKW}, for every $g\in G$, $\VV_{m}{(g)}\subseteq Z_{g}$ is a clopen set, hence all the elements of $\Ba{(Z_{G})}$ are clopens of $Z_{G}$. Moreover, since $\{\VV_m{(g)}\}_{g\in G}\seq\Ba{(Z_{G})}$, they also form a closed base for $Z_{G}$, and hence $\Ba{(Z_{G})}$ is a Wallman base of $Z_{G}$. We denote by
\begin{equation}\label{e:wG}
w_{G}Z_{G}
\end{equation}
the Wallman compactification $w_{\Ba{(Z_{G})}}Z_{G}$ given by $\Ba{(Z_{G})}$. The Boolean algebra of all clopens of $w_{G}Z_{G}$, written $\Cp{(w_{G}Z_{G})}$, can be identified with the elements of $\Ba{(Z_{G})}$:
\begin{equation}\label{e:clopens}
 \Ba{(Z_{G})} \cong \Cp{(w_{G}Z_{G})}.
\end{equation}

Let $\K{(Z_{G})}\seq\C{(Z_{G})}$ be the Boolean algebra of characteristic functions for the elements of $\Ba{(Z_{G})}$. Since each member of $\K{(Z_{G})}$ extends  uniquely to a member  of the Boolean algebra $\K{(w_{G}Z_{G})}$ of all characteristic functions in $\C{(w_{G}Z_{G})}$ by \cite[4G]{PorterWoods88},
we have the isomorphism of Boolean algebras
\begin{equation}\label{e:charfunctions}
 \K{(Z_{G})} \cong \K{(w_{G}Z_{G})}.
\end{equation}
In summary, any two Boolean algebras in (\ref{e:clopens}) and (\ref{e:charfunctions}) are isomorphic, and their dual Stone space is  (\ref{e:wG}).
\begin{remark}\label{r:beta0}It may happen that $w_{G}Z_{G}$ is strictly smaller than the \emph{Banaschewski compactification} of $Z_{G}$, the largest zero-dimensional compactification usually denoted $\beta_{0}Z_{G}$. See Section \ref{s:ex}.
\end{remark}
\begin{lemma}\label{l:freudenthal}With reference to the embedding \textup{(\ref{eq:mu-rep})}, the  uniform completion of the linear subspace of $\C{(Z_{G})}$ generated by $\K{(Z_{G})}$ contains $\mu(\widehat{G})$.
\end{lemma}
\begin{proof}Let $V$ be the Riesz space generated by $\mu(\widehat{G})\cup \K{(Z_{G})}$ in $\C{(Z_{G})}$.
\begin{claim}
 The function $1_{Z_{G}}$ is a strong unit for $V$.
\end{claim}
\begin{proof}
For each $g\in G$ the function $\widehat{g}$ is bounded in $X_{G}$. Since the image of the function $\widehat{g}\circ\lambda$ coincides with the image of $\widehat{g}$, each element $\mu(\widehat{g})\in\mu(\widehat{G})$ is bounded. The characteristic functions in $\K{(Z_{G})}$ are bounded functions by definition. Therefore each element of $V$ is bounded, because the Riesz space operations preserve the property of being bounded. Hence $1_{Z_{G}}$ is a strong unit for $V$.
\end{proof}
For each $g\in G$ let $\chi_{g}\in\K{(Z_{G})}$ be the characteristic function of $\Ss_{m}(g)$.
\begin{claim}\label{c:orthogonal}$\chi_{g}^\perp=\mu(\widehat{g})^\perp$ in $V$. 
\end{claim}
\begin{proof}[Proof of Claim \ref{c:orthogonal}] 
Let $v\in V$ such that $v\in\chi_{g}^{\perp}$. Hence $v(\p)\land\chi_{g}(\p)=0$ for every $\p\in Z_{G}$. As a consequence, if $\chi_{g}(\p)\neq0$, then $v(\p)=0$. Therefore, whenever $\mu(\widehat{g})(\p)\neq0$, we have $v(\p)=0$ by Remark \ref{r:zerosets}. Then $v(\p)\land\mu(\widehat{g})(\p)=0$ for all $\p\in Z_{G}$, whence $v\in\mu(\widehat{g})^{\perp}$. This proves the inclusion $\chi_{g}^\perp\seq\mu(\widehat{g})^\perp$.\\
For the other inclusion, suppose there exists $v\in V$ such that $v\in\mu(\widehat{g})^\perp\setminus\chi_{g}^\perp$. Then the supports $\Ss{(v)}$ and $\Ss{(\mu(\widehat{g}))}$ are disjoint, and there is $\p\in Z_{G}$ such that $|v(\p)|\land|\chi_{g}(\p)|\neq0$. By definition of $\chi_{g}$, $\chi_{g}^{-1}(1)=\Ss{(\chi_{g})} = \Ss_{m}{(g)}$ and therefore $\p$ is in $\Ss{(v)}\cap\Ss_{m}{(g)}$, which is open since $v$ is a continuous function on $Z_{G}$ and $\Ss_{m}{(g)}$ is a basic open subset of $Z_{G}$. Hence there exists a (nonempty) basic open set $\Ss_{m}{(h)}$ (with $0\neq h\in G$) such that $\p\in\Ss_{m}{(h)}\seq\Ss{(v)}\cap\Ss_{m}{(g)}$. Let $\m\in\Ss_{M}{(h)}$. Then $h\not\in\m$, whence $h\not\in\q$ for every $\q\in\lambda^{-1}(\m)$ and $\lambda^{-1}(\m)\seq\Ss_{m}{(h)}$. Since $\Ss_{m}{(h)}$ is disjoint from $\Ss{(\mu(\widehat{g}))}$, and $\Ss{(\mu(\widehat{g}))}=\lambda^{-1}(\Ss_{M}{(g)})$ by Remark \ref{r:zerosets}, $\q\not\in\lambda^{-1}(\Ss_{M}{(g)})$ for every $\q\in\lambda^{-1}(\m)$. Then $\m\in\VV_{M}{(g)}$ and $g\in\m$. This ensures $\Ss_{M}{(h)}\seq\VV_{M}{(g)}$, and hence $|h|\land|g|=0$. By primality of $\p$, either $h\in\p$ or $g\in\p$, in contradiction with $\p\in\Ss_{m}{(h)}\seq\Ss_{m}{(g)}$. This ensures $v\in\chi_{g}^\perp$, whence $\mu(\widehat{g})^\perp\seq\chi_{g}^\perp$.
\end{proof}
\begin{claim}\label{c:projectable}$V$ is projectable. 
\end{claim}
\begin{proof}[Proof of Claim \ref{c:projectable}] We need to show that each $v \in V$ induces a product splitting  $V\cong v^{\pperp}\times v^{\perp}$. 
By Claim \ref{c:orthogonal}, $\chi_{g}^\pperp=\mu(\widehat{g})^\pperp$ for all $g\in G$. But since $\chi_{g}$ is a component of $1_{Z_{G}}$ (because $\chi_{g}\vee (1_{Z_{G}}-\chi_{g})=1_{Z_{G}}$ and $\chi_{g}\wedge(1_{Z_{G}}-\chi_{g})=0$), there is an induced  splitting $V\cong \chi_{g}^\pperp \times \chi_{g}^{\perp}$. This shows that each element of $V$ that lies in the generating set  $\mu(\widehat{G})\cup \K{(Z_{G})}$  induces the splitting above. A routine induction argument now shows that the splitting property is preserved by sums, meets and joins, and products by real scalars, thus settling the claim.
\end{proof}
By the preceding claim, we may apply  Freudenthal's Spectral Theorem \ref{t:freudenthal} to $(V,1_{Z_{G}})$, and infer that each element of   $\mu(\widehat{G})$ is a $1_{Z_{G}}$-uniform limit of a sequence of elements in the linear subspace of $\C{(Z_{G})}$ generated by $\K{(Z_{G})}$. Since the norm induced by $1_{Z_{G}}$ on $\C{(Z_{G})}$ coincides with the supremum norm, this completes the proof.
\end{proof}
\begin{lemma}\label{l:ext}For each $g\in G$,  there exists a unique continuous extension of $\mu(\widehat{g}) \in \C{(Z_{G})}$ to an element $g^{\sharp}\in\C{(w_{G}Z_{G})}$. That is, $g^{\sharp}$ is the unique such element whose restriction to $Z_{G}$ is $\mu(\widehat{g})$. In symbols,
\begin{align}\label{eq:ext}
g^{\sharp}{\upharpoonright_{Z_{G}}}=\mu(\widehat{g}).
\end{align}
\end{lemma}
\begin{proof}Indeed, by Lemma \ref{l:freudenthal} there is a sequence $\{c_{i}\}_{i\geq 1}$ of linear combinations of elements of $\K{(Z_{G})}$ that converges uniformly to $\mu(\widehat{g})$.  Each member of $\K{(Z_{G})}$ extends  uniquely to a member  of $\K{(w_{G}Z_{G})}$ \cite[4G]{PorterWoods88}, and therefore each $c_{i}$ extends to a linear combination $k_{i}$ of  elements of  $\K{(w_{G}Z_{G})}$. It is now elementary to check that $\{k_{i}\}_{i\geq 1}$ is a Cauchy sequence in  $\C{(w_{G}Z_{G})}$ because $\{c_{i}\}_{i\geq 1}$ is one in $\C{(Z_{G})}$. Take $g^{\sharp}$ to be the limit of $\{k_{i}\}_{i\geq 1}$, which is of course a continuous function by the Uniform Limit Theorem. Finally, note that  $g^\sharp$ has property (\ref{eq:ext}) by construction, and is the unique  member of $\C{(w_{G}Z_{G})}$ with this property  because $Z_{G}$ is dense in its $w_{G}$-compactification, and the codomain of the functions --- namely, $\R$ --- is Hausdorff. 
\end{proof}
In light of Lemma \ref{l:ext}, the function
\begin{align}\label{eq:sharp}
\cdot^{\sharp}\colon G \hookrightarrow \C{(w_{G}Z_{G})}
\end{align}
that acts by $g\mapsto g^{\sharp}$ is injective. It is elementary that this embedding preserves the lattice and group structure of $G$, and is also unit-preserving. We have therefore proved:
\begin{theorem}\label{t:min-rep}Each $\U$-object $(G,u)$ has a representation into $\C{(w_{G}Z_{G})}$ as in \textup{(\ref{eq:sharp})}.\qed
\end{theorem}
\begin{definition}\label{d:PH}We  write $\PH{(G)}$ for the $\ell$-subgroup of $\C{(w_{G}Z_{G})}$ generated by 
 \begin{align}\label{eq:generators}
 G^{\sharp}\cup\K{(w_{G}Z_{G})}.
 \end{align} 
 We further write  
\begin{align}\label{eq:ph}
\pi \colon G \hookrightarrow \PH{(G)}
\end{align}
for the $\U$-monomorphism of $G$ into $\PH{(G)}$ obtained by restricting the codomain of (\ref{eq:sharp}) to $\PH{(G)}$.
\end{definition}
\section{Characterisation of the elements of $\PH{(G)}$}\label{s:el}
In this section we characterise the functions in $\C{(w_{G}Z_{G})}$ that lie in $\PH{(G)}$.
We begin by preparing two  lemmas.
\begin{lemma}\label{l:separation}There is a  homeomorphism $\MaxS{\PH{(G)}}\cong w_{G}Z_{G}$.
\end{lemma}
\begin{proof}Indeed, the characteristic functions $\K{(w_{G}Z_{G})}\seq \PH{(G)}$ separate the points of $w_{G}Z_{G}$, because the latter is zero-dimensional; now apply Yosida's Theorem \ref{t:yosida}. 
\end{proof}
\begin{remark}\label{r:known}The next lemma is a consequence of more general results, cf.\ \cite[Proposition 2.4]{HagerKimberMcGovern03}. We provide a proof here, for the reader's convenience.
\end{remark}
\begin{lemma}\label{l:products} Let $g\in G$, and let $\chi\in G$ be a component of the unit $u$. Let us identify $G$ with  its Yosida representation $\widehat{G}\seq\C{(X_{G})}$. The pointwise product $g\chi$ defined by $(g\chi)(x)=g(x)\chi(x)$ for each $x \in X_{G}$ is a continuous function, and hence an element of $\C{(X_{G})}$. Then $g\chi\in \widehat{G}$.
\end{lemma}
\begin{proof}(Skipping all trivialities, in this proof we identify isomorphism with equality  without further warning.) Since $\chi$ is a component of $u$ we have a product splitting $\widehat{G}=\chi^{\perp}\times\chi^{\perp\perp}$ ($\perp$ computed in $\widehat{G}$), and a corresponding disjoint union decomposition $X_{G}=A\sqcup B$, $A:=\chi^{-1}(0)$, $B:=\chi^{-1}(1)$, $A$ and $B$ disjoint clopens in $X_{G}$. Then, clearly, $\C{(X_{G})}=\C{(A)}\times\C{(B)}$, $\chi^{\perp}\seq \C{(A)}$, and  $\chi^{\perp\perp}\seq \C{(B)}$.  Now since $g\in \chi^{\perp}\times\chi^{\perp\perp}$, $g$ may be uniquely expressed as a sum $g_{1}+g_{2}$, $g_{i} \in \widehat{G}$, $g_{1}\in \chi^{\perp}$, $g_{2} \in \chi^{\perp\perp}$. Then $g$ and $g_{2}$ agree over $B$, so that $g\chi=g_{2}\chi=g_{2}\in \widehat{G}$, and the lemma is proved.
\end{proof}
\begin{remark}
Let $0\leq g\in G$, and let $\chi\in G$ be a component of the unit
$u$. Identifying $G$ with its Yosida representation
$\widehat{G}\seq\C{(X_{G})}$, we notice that the function $g$ is
bounded on the support of the characteristic function
$\chi$. Therefore, there exists a (unique minimal) integer $n\geq0$
such that $g\leq n\chi$ holds on the support of $\chi$, and hence
$g\chi=g\wedge n\chi$ holds in $G$. This yields  an explicit
representation of the product $g\chi$ discussed in Lemma
\ref{l:products}, using only the operations of $G$. Any element $g\in
G$, indeed, can be written as the difference $g^{+}-g^{-}$ between its
\emph{positive part} $g^{+}:=g\vee 0$ and its \emph{negative part}
$g^{-}:=(-g)\vee 0$, with $0\leq g^{+},g^{-}\in G$. As a consequence,
there exist two (unique minimal) integers $n_{+},n_{-}\geq0$ such that
$g\chi = (g^{+}\wedge n_{+}\chi) - (g^{-}\wedge n_{-}\chi)$.
 
In the following, we use the product $g\chi$ for brevity, but each
such occurrence may be replaced by the equivalent expression
$(g^{+}\wedge n_{+}\chi) - (g^{-}\wedge n_{-}\chi)$.
\end{remark}
By a \emph{partition of unity} in a $\U$-object $(G,u)$ we mean in this paper a finite family of  non-zero elements $P:=\{\chi_{i}\}_{i=1}^{l}$ of $G$ such that $\sum_{i=1}^{l}\chi_{i}=u$, and $\chi_{i}\wedge\chi_{j}=0$ whenever $i\neq j$. It is elementary that each $\chi_{i}$ is a component of $u$. It follows that, in the Yosida representation $\widehat{G}$ of $G$, each $\widehat{\chi}_{i}$ is a characteristic function.

We can now prove:
\begin{theorem}\label{t:decomposition}For each $e\in \C{(w_{G}Z_{G})}$, the following are equivalent.
\begin{enumerate}
\item\label{i:decomp1} $e\in \PH{(G)}$.
\item\label{i:decomp2} There exists a partition of unity $\chi_{1},\ldots,\chi_{l}$ in $\PH{(G)}$  --- equivalently, in $\C{(w_{G}Z_{G})}$ ---  along with elements $a_{1},\ldots,a_{l}\in G$, such that 
\begin{align}\label{eq:decomp}
e=\sum_{i=1}^{l}a_{i}^{\sharp}\chi_{i},
\end{align}
where $\sharp$ is the embedding \textup{(\ref{eq:sharp})}, and $a_{i}^{\sharp}\chi_{i}$ denotes the pointwise product of $a_{i}^{\sharp}$ and $\chi_{i}$ in $\C{(w_{G}Z_{G})}$.
\end{enumerate}
\end{theorem}
\begin{proof}First, let us explicitly note that $\PH{(G)}$ and  $\C{(w_{G}Z_{G})}$ have  the same collection of partitions of unity because $\K{(w_{G}Z_{G})}\seq \PH{(G)}$. 

\smallskip \noindent (\ref{i:decomp1})$\Rightarrow$(\ref{i:decomp2})\
Recall that $\PH{(G)}$ is the $\ell$-subgroup of $\C{(w_{G}Z_{G})}$ generated by the set (\ref{eq:generators}). Hence by the elementary theory of lattice-groups we can write $e$ as
\[ 
\bigwedge_{i\in I} \bigvee_{j \in J} \left(g_{ij}^{\sharp} + c_{ij}k_{ij}\right), 
\]
where $I$ and $J$ are finite sets of indices, at least one of which is non-empty, $g_{ij} \in G$, $c_{ij} \in \Z$ and
$k_{ij} \in \K{(w_{G}Z_{G})}$.  Now for each $k_{ij}$,
we obtain associated clopen subsets of $w_{G}Z_{G}$, namely their supports
and their complements. This (necessarily non-empty) collection of clopens obviously covers $w_{G}Z_{G}$. It is elementary that we can refine this cover into a
finite partition $\{D_m\}_{m \in M}$ of clopens of the space  by taking
intersections and set-theoretic differences.

On each $D_m$, each $k_{ij}$ is constant --- either zero or one --- by construction. Let us define the element $\delta_{ij}^{m}\in G$ by setting
\begin{align}\label{eq:delta}
\delta_{ij}^{m}:=\begin{cases}u &\text{if $D_{m}\seq k_{ij}^{-1}(1)$},\\
0 &\text{if $D_{m}\seq k_{ij}^{-1}(0)$}.\end{cases}
\end{align}
Now consider the element of $\PH{(G)}$
\begin{align*} 
e_{m}:=\bigwedge_{i\in I} \bigvee_{j \in J} \left(g_{ij}^{\sharp} + c_{ij}(\delta_{ij}^{m})^\sharp\right).
\end{align*}
Observe that \emph{the function $e_{m}$ agrees over $D_{m}$ with the function $e$}, for each $m \in M$. This follows immediately from our definition of $\delta_{ij}^{m}$ in (\ref{eq:delta})  above.
Moreover, since $\sharp$ is an $\ell$-homomorphism, we have
\[ 
e_{m}=\left(\bigwedge_{i\in I} \bigvee_{j \in J} \left(g_{ij} + c_{ij}\delta_{ij}^{m}\right)\right)^{\sharp}=a_{m}^{\sharp},
\]
where  $a_{m}:=\bigwedge_{i\in I} \bigvee_{j \in J} \left(g_{ij} + c_{ij}\delta_{ij}^{m}\right)\in G$. Hence, if we let $\chi_m$ be the characteristic
function of $D_m$, we conclude 
\[ e = \sum_{m\in M} a_m^{\sharp}\chi_m,\]
as was to be shown.

\smallskip \noindent (\ref{i:decomp2})$\Rightarrow$(\ref{i:decomp1})\ This follows at once from Lemmas \ref{l:separation} and \ref{l:products}.

\end{proof}

 \section{Construction of the projectable hull}\label{s:ph}
Our final aim is to show that the embedding \textup{(\ref{eq:sharp})} provides a description of the projectable hull of $G$. This is  our main result:

\begin{theorem}\label{t:ph}
For any $\U$-object $(G,u)$, the embedding $\pi \colon G \hookrightarrow \PH{(G)}$ as in \textup{(\ref{eq:ph})} is the
projectable hull of $G$.
\end{theorem}
\begin{proof} The proof that $\PH{(G)}$ is projectable is identical to that of Claim
\ref{c:projectable}.

To prove that the map $\pi$ is an essential extension, we verify (\ref{i:ess3}) in Lemma \ref{l:ess}. Pick $0<e\in\PH{(G)}$, and express it as $e=\sum_{i=1}^{l}a_{i}^{\sharp}\chi_{i}$ by Theorem \ref{t:decomposition}, for a partition of unity $\{\chi_{i}\}_{i=1}^{l}$ in $\PH{(G)}$ and elements $\{a_{i}\}_{i=1}^{l}$ in $G$. Since $e>0$, we must have $a_{i}^{\sharp}\chi_{i}\geq 0$ for each $i$, and $a_{i_{0}}^{\sharp}\chi_{i_{0}}> 0$ for some $i_{0}$. It is enough to show that there is $h\in G$ such that $0<h^{\sharp}\leq n a_{i_{0}}^{\sharp}\chi_{i_{0}}$, for then $0<h^{\sharp}\leq ne$ follows easily.  By \eqref{e:charfunctions}, we can identify the characteristic functions in $\K{(w_{G}Z_{G})}$ with the elements of $\K{(Z_{G})}$.
Set $a:=a_{i_{0}}$ and $\chi:=\chi_{i_{0}}\land\chi_{a}$, where $\chi_{a}\in\K{(Z_{G})}$ is the characteristic function of $\Ss_{m}(a)$.
%Since, by \cite[10.2.1]{BKW}, the clopens of $Z_{G}$ are of the form $\VV_{m}{(g)}$ for $0\leq g\in G$, and since $Z_{G}$ and $w_{G}Z_{G}$ have essentially the same clopens by (\ref{eq:beta0Cp}),  there is $0<g\in G$ such that $(g^{\sharp})^{\perp\perp} =\chi^{\perp\perp}$, from which it follows at once that  the support of $g^{\sharp}$ is contained in that of $\chi$. 
Since the support $\Ss{(\chi)}=\chi^{-1}(1)$ is clopen and nonempty (because $a^{\sharp}\chi\neq0$), there exists $0<g\in G$ such that the basic open set $\Ss_{m}{(g)}$ is nonempty and contained in $\Ss{(\chi)}\seq\Ss_{m}{(a)}$. Hence $h:=g\land a>0$, and $\Ss{(h^{\sharp})}\seq\Ss_{m}{(h)}\seq\Ss{(\chi)}$. 
%Now  $\Ss{(\chi)}$ is closed, hence compact, and therefore by the Extreme Value Theorem there is an integer $n\geq 1$ such that $n\chi\geq g^{\sharp}$. 
We now claim that $h^{\sharp}\leq a^{\sharp}\chi$. 
%Since the support of $g^{\sharp}$ is contained in that of $\chi$, the support of $a^{\sharp}\wedge g^{\sharp}$ is contained in that of $a^{\sharp}\chi$, and 
It is enough to prove that the inequality holds for a point $x\in w_{G}Z_{G}$ in the support of $\chi$, where $\chi(x)=1$. Since $\sharp$ is an $\ell$-homomorphism, $h^{\sharp}=a^{\sharp}\wedge g^{\sharp}>0$. If $a^{\sharp}(x)\leq g^{\sharp}(x)$, then $h^{\sharp}(x)=a^{\sharp}(x)$, and the inequality holds. Otherwise, we have $h^{\sharp}(x)=g^{\sharp}(x)< a^{\sharp}(x)$. This settles the claim and
%We now set $h:=a\wedge g$. Then, since $\sharp$ is an $\ell$-homomorphism, $h^{\sharp}=a^{\sharp}\wedge g^{\sharp}\leq n a^{\sharp}\chi$. Moreover, $h>0$. Indeed, the support of $g$ is nonempty and contained in the support of $a$ by our choice of $g$. Hence $\chi_{a}^{\perp}\seq\chi_{g}^{\perp}$. Since $\chi_{a}^{\perp}=(a^{\sharp})^{\perp}$ and $\chi_{g}^{\perp}=(g^{\sharp})^{\perp}$ by Claim \ref{c:orthogonal}, $a^{\sharp}\not\in (g^{\sharp})^\perp$, and $a^{\sharp}\wedge g^{\sharp}\neq 0$. This 
completes the proof that $\pi$ is essential.

To show $\PH{(G)}$ is a hull, it suffices to show that given the (unital)
essential embedding $\iota$ into $H$ there exists an (automatically essential and unital)
embedding $\varphi$ making the diagram below commute:
\[
 \begin{tikzcd}\tag{*}\label{d:phi}
(G,u) \arrow[hookrightarrow]{r}{\pi}
   \arrow[hookrightarrow]{dr}{\iota}& 
   (\PH{(G)},u^{\sharp}) \arrow[hookrightarrow,dashed]{d}{\varphi} \\ 
   &(H,v)
 \end{tikzcd}
\]
We  define a function $\varphi\colon \PH{(G)}\to H$ as follows. First we set
\begin{align}\label{eq:phionG}
\varphi(a^{\sharp}):=\iota(a), \text{ for each } a \in G.
\end{align}
Since $H$ is projectable, for each $g\in G$ there is a unique component $\sigma_{g}\in H$ of the unit $v$ that satisfies $\iota(g)^{\perp\perp}=\sigma_{g}^{\perp\perp}$,
where $\perp$ is computed in $H$. By \eqref{e:charfunctions}, we identify the components of the unit $u^{\sharp}$ in $\PH{(G)}$ with the characteristic functions in $\K{(Z_{G})}$. Let $\chi_{g}\in\PH{(G)}$ be the characteristic function of $\Ss_{m}(g)$ for $g\in G$. The analogous result of Claim \ref{c:orthogonal} for $\PH{(G)}$ ensures that $\chi_{g}^{\perp}=\pi(g)^{\perp}$ in $\PH{(G)}$. We set
\begin{equation}\label{eq:phionchig}
 \varphi(\chi_{g}):= \sigma_{g}.
\end{equation}
Let $\K{(H)}$ be the Boolean algebra of components of the unit $v$ in $H$.
\begin{claim}\label{c:phihom}
 The map $\varphi\colon\{\chi_{g}\}_{g\in G}\to\K{(H)}$ defined as in \eqref{eq:phionchig} is a lattice homomorphism.
\end{claim}
\begin{proof}
 Since $\iota$ is a unital homomorphism, trivially $\varphi(\chi_{0})=0$ and $\varphi(\chi_{u})=v$. Let $g_{1},g_{2}\in G$. It is easy to check that $\chi_{g_{1}}\land\chi_{g_{2}}=\chi_{g_{1}\land g_{2}}$ and $\chi_{g_{1}}\vee\chi_{g_{2}}=\chi_{g_{1}\vee g_{2}}$. Let $\varphi(\chi_{g_{1}})=\sigma_{g_{1}}$, $\varphi(\chi_{g_{2}})=\sigma_{g_{2}}$, $\varphi(\chi_{g_{1}\land g_{2}})=\sigma_{g_{1}\land g_{2}}$ and $\varphi(\chi_{g_{1}\vee g_{2}})=\sigma_{g_{1}\vee g_{2}}$ as in \eqref{eq:phionchig}. Hence, by \eqref{e:ppolmeet} and \eqref{e:ppoljoin},
 \begin{align*}
  \sigma_{g_{1}\land g_{2}}^{\pperp}&=\iota(g_{1}\land g_{2})^{\pperp} = (\iota(g_{1})\land \iota(g_{2}))^{\pperp} =\\
  &= \iota(g_{1})^{\pperp}\cap \iota(g_{2})^{\pperp} = \sigma_{g_{1}}^{\pperp}\cap \sigma_{g_{2}}^{\pperp}= (\sigma_{g_{1}}\land\sigma_{g_{2}})^{\pperp},
 \end{align*}
 and
 \begin{align*}
  \sigma_{g_{1}\vee g_{2}}^{\pperp}&=\iota(g_{1}\vee g_{2})^{\pperp} = (\iota(g_{1})\vee \iota(g_{2}))^{\pperp} =\\
  &= (\iota(g_{1})^{\pperp}\cup \iota(g_{2})^{\pperp})^{\pperp} = (\sigma_{g_{1}}^{\pperp}\cup\sigma_{g_{2}}^{\pperp})^{\pperp} = (\sigma_{g_{1}}\vee\sigma_{g_{2}})^{\pperp}.
 \end{align*}
This ensures $\sigma_{g_{1}\land g_{2}}=\sigma_{g_{1}}\land\sigma_{g_{2}}$ and $\sigma_{g_{1}\vee g_{2}}=\sigma_{g_{1}}\vee\sigma_{g_{2}}$. Therefore $\varphi(\chi_{g_{1}}\land\chi_{g_{2}})=\varphi(\chi_{g_{1}\land g_{2}})=\varphi(\chi_{g_{1}})\land\varphi(\chi_{g_{2}})$ and $\varphi(\chi_{g_{1}}\vee\chi_{g_{2}})=\varphi(\chi_{g_{1}\vee g_{2}})=\varphi(\chi_{g_{1}})\vee\varphi(\chi_{g_{2}})$, and the claim is proved.  
\end{proof}
\noindent
The previous Claim \ref{c:phihom}, together with \cite[Lemma 1 in Section V.4]{BalbesDwinger74}, ensures that $\varphi$  extends  uniquely to a homomorphism $\varphi\colon \K{(Z_{G})}\to\K{(H)}$ of Boolean algebras. Specifically, since the Boolean algebra $\K{(Z_{G})}$ can be generated by the set $\{\chi_{g}\}_{g\in G}$ using the negation $u^{\sharp}-$ and the meet operation $\land$, the image of a general element of $\K{(Z_{G})}$ under $\varphi$ can be inductively described as follows. Let $\chi_{1}, \chi_{2}, \alpha, \beta_{1}, \beta_{2}\in \K{(Z_{G})}$ such that $\chi_{1} = u^{\sharp}-\alpha$ and $\chi_{2}=\beta_{1}\land\beta_{2}$. Then
\begin{equation}\label{eq:phionK}
 \varphi(\chi_{1})=v-\varphi(\alpha)\quad\text{and}\quad
 \varphi(\chi_{2})=\varphi(\beta_{1})\land\varphi(\beta_{2}).
\end{equation}
\noindent
Finally, for a general $e \in \PH{(G)}$, we first write $e$ as in (\ref{eq:decomp}) using Theorem \ref{t:decomposition}, and then, using (\ref{eq:phionG}--\ref{eq:phionK}), we set
\begin{align}\label{eq:phigeneral}
\varphi(e):=\sum_{i=1}^{l}\varphi(a_{i}^{\sharp})\varphi(\chi_{i}).
\end{align}
Since each product $\varphi(a_{i}^{\sharp})\varphi(\chi_{i})$  is an element of $H$ by Lemma \ref{l:products}, $\varphi(e)$ as in (\ref{eq:phigeneral}) is an element of $H$.

We next verify that $\varphi$  is a well-defined function. Given a decomposition  (\ref{eq:decomp}) of $e\in \PH{(G)}$ as in Theorem \ref{t:decomposition}, suppose $e=\sum_{j=1}^{t}b_{j}^{\sharp}\xi_{j}$ is another such decomposition. It suffices to show that 
\begin{align}\label{eq:show}
\sum_{i=1}^{l}\varphi(a_{i}^{\sharp})\varphi(\chi_{i})=\sum_{j=1}^{t}\varphi(b_{j}^{\sharp})\varphi(\xi_{j}).
\end{align}
It is elementary to verify that the set $\{\chi_{i}\wedge \xi_{j}\mid \chi_{i}\wedge \xi_{j}\neq 0 \}$
forms a partition of unity that refines both $\{\chi_{i}\}_{i=1}^{l}$
and $\{\xi_{j}\}_{j=1}^{t}$; that is,  each $\chi_i$ and $\xi_j$ is a sum (or join, by pairwise disjointness) of elements $\chi_{i'}\wedge \xi_{j'}\neq 0$.
It follows that $e$ can  be expressed
in two ways as
\begin{align*}
e=\sum a_{i}^{\sharp}(\chi_{i}\wedge \xi_{j})=\sum b_{j}^{\sharp}(\chi_{i}\wedge \xi_{j}),
\end{align*}
and hence $a_{i}^{\sharp}(\chi_{i}\wedge \xi_{j})=b_{j}^{\sharp}(\chi_{i}\wedge \xi_{j})$ for all $i,j$.

\noindent
In the simplest case, we must prove:
\begin{claim} 
Assuming $\chi\in\K{(Z_{G})}$,  suppose $a^{\sharp}\chi =
b^{\sharp}\chi$. Then $\varphi(a^{\sharp}\chi) = \iota(a)\varphi(\chi)=\iota(b)\varphi(\chi) = \varphi(b^{\sharp}\chi)$.
\end{claim}
\begin{proof}
Let $\nu_{\PH{(G)}}\colon\Pol{\PH{(G)}}\to\Pol{\pi(G)}$ and $\nu_{H}^{-1}\colon\Pol{\iota(G)}\to\Pol{H}$ be the isomorphisms of Boolean algebras given in Lemma \ref{l:ess}. Then also $\nu_{\pi}:=\pi^{-1}\circ\nu_{\PH{(G)}}\colon\Pol{\PH{(G)}}\to\Pol{G}$ and $\nu_{\iota}:=\nu_{H}^{-1}\circ\iota\colon\Pol{G}\to\Pol{H}$ are isomorphisms of Boolean algebras. Since $\chi^{\perp_{\PH{(G)}}}$ is a
polar, there exists a unique polar $P_{\chi}\in\Pol{G}$ such that $P_{\chi}^{\perp_{G}}=\nu_{\pi}(\chi^{\perp_{\PH{(G)}}})$.
From $a^{\sharp}\chi =
b^{\sharp}\chi$, we have $\pi(a-b) = (a^{\sharp} -
b^{\sharp}) \in \chi^{\perp_{\PH{(G)}}}$. Hence $a-b \in
P_{\chi}^{\perp_G}$ and $\iota(a)-\iota(b)=\iota(a-b) \in \iota(P_{\chi})^{\perp_{\iota(G)}} =
\iota(P_{\chi})^{\perp_H} \cap \iota(G)$. 
To complete the proof we show $\iota(P_{\chi})^{\perp_H}=\varphi(\chi)^{\perp_H}$, whence $(\iota(a)-\iota(b))\in\varphi(\chi)^{\perp_H}$ and $\iota(a)\varphi(\chi)=\iota(b)\varphi(\chi)$. \\
We proceed by induction on the structure of $\chi$ using \eqref{eq:phionK}. 
In the basic case we suppose $\chi:=\chi_{g}$, for $g\in G$. 
Then $\chi_{g}^{\perp_{\PH{(G)}}}=\pi(g)^{\perp_{\PH{(G)}}}$ by Claim \ref{c:orthogonal}, whence $P_{\chi}^{\perp_{G}} = \nu_{\pi}(\pi(g)^{\perp_{\PH{(G)}}}) = \pi^{-1}(\nu_{\PH{(G)}}(\pi(g)^{\perp_{\PH{(G)}}})) = \pi^{-1}(\pi(g)^{\perp_{\PH{(G)}}}\cap\pi(G)) = g^{\perp_{G}} = (g^{\pperp_{G}})^{\perp_{G}}$, and $P_{\chi} = g^{\pperp_{G}}$. 
By definition of $\varphi$ on $\{\chi_{g}\}_{g\in G}$, we have $\varphi(\chi)^{\perp_{H}} = \iota(g)^{\perp_{H}} = \iota(g^{\pperp_{G}})^{\perp_{H}}$, where the last equality is a direct consequence of the essentiality of $\iota$ (see \cite[\S 2]{Conrad73}). \\
We now consider the case $\chi := u^{\sharp} - \alpha$, and suppose $\iota(P_{\alpha})^{\perp_H}=\varphi(\alpha)^{\perp_H}$. 
Since $\chi$ and $\alpha$ are components of unity, $\chi^{\perp_{\PH{(G)}}} = \alpha^{\pperp_{\PH{(G)}}}$, whence $P_{\chi}^{\perp_{G}} = P_{\alpha}^{\pperp_{G}}$ and $\iota(P_{\chi})^{\perp_{H}} = \iota(P_{\alpha})^{\pperp_{H}}$, because both $\nu_{\pi}$ and $\nu_{\iota}$ are  isomorphisms of Boolean algebras. Since also $\varphi\colon\K{(Z_{G})}\to\K{(H)}$ is an isomorphism of Boolean algebras, $\varphi(\chi)=v-\varphi(\alpha)$. 
Therefore, $\iota(P_{\chi})^{\perp_H}=\iota(P_{\alpha})^{\pperp_H}=\varphi(\alpha)^{\pperp_H}=\varphi(\chi)^{\perp_H}$. \\
The case $\chi:=\beta_{1}\land\beta_{2}$ is proved in a similar way. Suppose $\iota(P_{\beta_{1}})^{\perp_H}=\varphi(\beta_{1})^{\perp_H}$ and $\iota(P_{\beta_{2}})^{\perp_H}=\varphi(\beta_{2})^{\perp_H}$. 
By \eqref{e:ppolmeet}, $\chi^{\pperp_{\PH{(H)}}}=(\beta_{1}\land\beta_{2})^{\pperp_{\PH{(H)}}}=\beta_{1}^{\pperp_{\PH{(H)}}}\cap\beta_{2}^{\pperp_{\PH{(H)}}}$, whence $\chi^{\perp_{\PH{(H)}}} = (\beta_{1}^{\pperp_{\PH{(H)}}}\cap\beta_{2}^{\pperp_{\PH{(H)}}})^{\perp_{\PH{(H)}}} = \beta_{1}^{\pperp_{\PH{(H)}}}\vee\beta_{2}^{\pperp_{\PH{(H)}}}$. As a consequence, $P_{\chi}^{\perp_{\PH{(H)}}} = P_{\beta_{1}}^{\perp_{\PH{(H)}}} \vee P_{\beta_{2}}^{\perp_{\PH{(H)}}}$ and $\iota(P_{\chi})^{\perp_{H}} = \iota(P_{\beta_{1}})^{\perp_{H}} \vee \iota(P_{\beta_{2}})^{\perp_{H}}$. 
Since $\varphi(\chi) = \varphi(\beta_{1})\land\varphi(\beta_{2})$, we have $\varphi(\chi)^{\pperp_{\PH{(H)}}} = \varphi(\beta_{1})^{\pperp_{\PH{(H)}}}\cap\varphi(\beta_{2})^{\pperp_{\PH{(H)}}}$, whence $\varphi(\chi)^{\perp_{\PH{(H)}}} = \varphi(\beta_{1})^{\perp_{\PH{(H)}}}\vee\varphi(\beta_{2})^{\perp_{\PH{(H)}}} = \iota(P_{\beta_{1}})^{\perp_{H}} \vee \iota(P_{\beta_{2}})^{\perp{H}} = \iota(P_{\chi})^{\perp_{H}}$. The claim is settled.
\end{proof}
Since $\varphi\colon \K{(Z_{G})}\to\K{(H)}$ is a homomorphism of Boolean algebras, $\varphi$ preserves
partitions of unity and $\varphi(\chi \wedge \xi) = \varphi(\chi) \wedge
\varphi(\xi)$. Therefore
\begin{align*}
\sum \varphi(a_{i}^{\sharp})\varphi(\chi_{i}\wedge \xi_{j})&=\sum \varphi(b_{j}^{\sharp})\varphi(\chi_{i}\wedge \xi_{j}), \\
\sum_j \varphi(a_{i}^{\sharp})\varphi(\chi_{i}\wedge \xi_{j}) &= \varphi(a_{i}^{\sharp})\varphi(\chi_{i}) \mbox{ for all  $i$'s,} \\
\sum_i \varphi(b_{j}^{\sharp})\varphi(\chi_{i}\wedge \xi_{j}) &= \varphi(b_{j}^{\sharp})\varphi(\xi_{j}) \mbox{ for all $j$'s.} 
\end{align*}

These together prove (\ref{eq:show}).

The map $\varphi$ makes the diagram (\ref{d:phi}) commute by
construction. To show that it is  an $\ell$-homomorphism one argues as follows. Given $e+f \in \PH{(G)}$, to prove $\varphi(e + f) =
\varphi(e) + \varphi(f)$, we first take decompositions of $e+f$, $e$ and
$f$ as in (\ref{eq:decomp}) of Theorem \ref{t:decomposition}. We then pick a joint refinement of the three partitions of unity involved.  We finally proceed as in the preceding argument   that shows  $\varphi$ is well-defined.
We omit the  elementary  details. The argument for the remaining operations is analogous.

To show $\varphi$ is  injective, consider $e \neq f \in
\PH{(G)}$. Using again decompositions as in (\ref{eq:decomp}) of  Theorem \ref{t:decomposition}, and a
common refinement of the associated partitions, we see that $e$ and
$f$ must differ on some element of the common refinement. Injectivity
of $\varphi$ then follows at once from the injectivity of $\iota$ and $\pi$.
\end{proof}
\section{Examples}\label{s:ex}
We close the paper with two examples. Example \ref{ex:parab} shows that the space $Z_{G}$ of minimal primes can fail to be compact even though its $G$-indexed compactification $w_{G}Z_{G}$ is actually $\beta_{0}Z_{G}$, the largest zero-dimensional compactification of $Z_{G}$. This happens when the base of clopens of $Z_{G}$ indexed by elements of $G$ is not a Boolean algebra --- because it fails to be closed under complements --- and yet it is large enough that the Boolean algebra it generates consists of all clopens of $Z_{G}$.  We note in passing that the $\ell$-group in this first example can be shown to be finitely generated.  Example \ref{ex:ref}, by contrast, is an instance where the $G$-indexed compactification $w_{G}Z_{G}$ is strictly smaller than $\beta_{0}Z_{G}$. Here $G$ is of the form $\C({X)}$ for $X$ a compact Hausdorff space. 
\begin{example}\label{ex:parab}
Let $\mathscr{M}$ be the set of continuous and piecewise linear functions $f\colon[0,1]^{2}\to\R$ with integer coefficients. This means that $f$ is continuous, and that there are finitely many triplets of integers $(a_{i},b_{i}, c_{i}) \in \Z^{3}$ such that for all $(x,y)\in[0,1]^{2}$ we have  $f(x,y) = a_{i}x+b_{i}y+c_{i}$ for some $i$. When equipped with the pointwise operations of minimum, maximum, and addition, $\mathscr{M}$ is an $\ell$-group. The function identically equal to $1$ is a strong unit.

 Let
 \begin{equation*}
P:=\{(x,y)\in[0,1]^{2}\mid y=x^{2}\}\cup\{(x,y)\in[0,1]^{2}\mid y=0\},
 \end{equation*}
 and let $(G,1_{P})$ be the unital $\ell$-group obtained by restricting each element of $\mathscr{M}$ to ${P}$, where $1_{P}\colon P \to \R$ is the function constantly equal to $1$. We now show that
  $Z_{G}:=\MinS{G}$ fails to be compact. For this,  consider the projection function $\pi_{y}\colon(x,y)\mapsto y$, along with  the set of functions $h_{n}$ ($0<n\in\N$) which are equal to $nx-1$ on the segment $\{(x,y)\in[0,1]^2\mid y=0\text{ and }x\in[\frac{1}{n},1]\}$ and to $0$ elsewhere. These are elements of $G$.
  Then $\{\Ss_{m}(\pi_{y})\}\cup\{\Ss_{m}(h_{n})\}_{n>0}$ is an open cover of $Z_{G}$ without a finite subcover. According to Lemma \ref{l:conmar}, $G$ is not complemented. Indeed, by piecewise linearity and continuity, there can be no $f\in \mathscr{M}$ that restricts to $P$ so as to satisfy $\pi_{y}^{\pperp}=f^{\perp}$. By the same lemma, $\{\VV_{m}(g)\}_{g\in G}$ 
  is not a Boolean algebra. Indeed,  it does not contain the complement $\Ss_{m}(\pi_{y})$ of $\VV_{m}(\pi_{y})$. In particular, $\{\VV_{m}(g)\}_{g\in G}$ is not the collection of all clopens of $Z_{G}$, because $\Ss_{m}(\pi_{y})$ is clopen --- it is both a 
  basic open set and the intersection of the 
  closed sets in $\{\VV_{m}(h_{n})\}_{n>0}$. However, the Boolean algebra $\Ba{(Z_G)}$ generated by $\{\VV_{m}(g)\}_{g\in G}$ is the collection of \emph{all} clopens of $Z_{G}$, as can be seen by direct inspection. Therefore $\MaxS{\PH{(G)}}\cong w_{G}Z_{G}$ is homeomorphic to $\beta_{0}Z_{G}$, the maximal zero-dimensional compactification of $Z_{G}$.
\end{example}

\begin{example}\label{ex:ref}
Let $D$ be the discrete space with $|D| =
\aleph_1$, let $\alpha D = D \cup \{\alpha\}$ be the one-point
compactification of $D$, and let $G = \C{(\alpha D)}$ so that $\MaxS{G} \cong \alpha
D$.  Then $Z_{G}:=\MinS{G}$ can be identified with the collection of fixed ultrafilters on $D$
together with the free ultrafilters on $D$ which contain a countable
subset.
The basic closed sets of $Z_{G}$, of the form
\[\{\mathscr{U} \in Z_{G} \mid g^{-1}(0)\in \mathscr{U}\},\]
are actually clopen.  We show that the Boolean algebra
generated by these basic clopen sets is {\em not} the
Boolean algebra of all clopens, but a proper subalgebra. For, let
$D_1$ and $D_2$ be elements of a partition of $D$ with $|D_i| = \aleph_1$, $i=1,2$. Then we have an induced partition of $Z_{G}$ into two clopens:
\[
Z_{G} = \{\mathscr{U} \mid \exists C \in \mathscr{U} \mbox{ with }
C \subseteq D_1\} \biguplus \{\mathscr{U} \mid \exists C \in \mathscr{U} 
\mbox{ with } C \subseteq D_2\}.
\] 
Neither summand in the partition above is a
basic closed set of $Z_{G}$. Therefore $\MaxS{\PH{(G)}}\cong
w_{G}Z_{G} \not\cong \beta_0 Z_{G}$.
\end{example}
\begin{remark}\label{r:c(x)} In  case $G=\C{(X)}$ for some  compact Hausdorff space $X$, the main result of this paper  is related to several others  in the literature, and especially to the construction of zero-dimensional covers of $X$. Postponing fuller accounts to further research, we mention at least the following important 
connection. In \cite{Vermeer84}, Vermeer constructs the 
basically disconnected  cover of $X$ (in the more general case of  a completely regular Hausdorff space) through an inverse limit \cite[Theorem 4.4]{Vermeer84}. The inverse limit is defined by means of an auxiliary space  $\Lambda_1\,(X)$ --- itself a cover of $X$ --- for which see in particular \cite[Theorem 3.5]{Vermeer84}. Now, the 
results by Hager et al.\ in \cite[Section 3]{HagerKimberMcGovern03} entail amongst other things that, in our notation,  there is a  homeomorphism $\MaxS{\PH{(\C{(X)})}}\cong \Lambda_1(X)$. By   Lemma \ref{l:separation} we conclude
 $\Lambda_1\,(X) \cong w_{\C{(X)}}Z_{\C{(X)}}$. We thus obtain a description of Vermeer's cover $\Lambda_1\,(X)$ as the Wallman compactification of the \emph{minimal} spectral space of $\C{(X)}$ described in (\ref{e:wG}--\ref{e:charfunctions}) above. Contrast Vermeer's \cite[Theorem 3.5]{Vermeer84}, which describes  $\Lambda_1\,(X)$ using regularised  zero sets in  the \emph{maximal} spectral space $X$ of $\C{(X)}$. Further,  the space $Z_{\C{(X)}}$ is compact if, and only if, $X$ is \emph{cozero-complemented}, i.e.\ the lattice of cozero sets of $X$ is complemented; this is well-known and also follows  from Lemma \ref{l:conmar}. In this 
 case we  have $w_{\C{(X)}}Z_{\C{(X)}}\cong Z_{\C{(X)}}$.  
 Consequently, $\Lambda_1\,(X)\cong Z_{\C{(X)}}$ --- when $X$ is cozero-complemented, Vermeer's  $\Lambda_1\,(X)$ is the minimal spectral space of $\C{(X)}$, and further work  reveals Vermeer's cover $\Lambda_1\,(X)\to X$ in \cite{Vermeer84} as the ``push-up'' map $\lambda\colon Z_{\C{(X)}}\to X$ in  
 (\ref{eq:lambda}--\ref{eq:action}).
\end{remark}
\part*{Acknowledgements.}
\noindent We gratefully acknowledge  partial support by the Italian FIRB ``Futuro in Ricerca'' grant RBFR10DGUA. The grant partially supported a visit of R. N. Ball to the Universit\`a degli Studi di Milano, Italy, during which the fundamental ideas of the present paper were developed. 
We would also like to express our thanks to an anonymous referee for detecting a serious blunder in an earlier version of this paper, and for providing us with  Example \ref{ex:ref}. The same referee pointed out to us the relevance of  \cite{Vermeer84} and \cite{HagerKimberMcGovern03} for our results (cf.\ Remark \ref{r:c(x)}).
 Finally, we are grateful to S. J. van Gool for pointing out to us that the use of  the standard extension result  \cite[Lemma 1 in Section V.4]{BalbesDwinger74} could shorten the proof of Theorem \ref{t:ph}.

\bibliographystyle{plain}
\bibliography{biblio}

\end{document}